\documentclass[12pt]{amsart}

\usepackage[left=1in, right=1in, top=1in, bottom=1in]
{geometry}
\usepackage{amsfonts,amssymb,graphicx,amsmath,amsthm, float, cancel, comment,color}

\theoremstyle{plain}
\newtheorem{theorem}{Theorem}
\newtheorem{lemma}[theorem]{Lemma}

\newtheorem{definition}[theorem]{Definition}
\newtheorem{conjecture}[theorem]{Conjecture}
\newtheorem{corollary}[theorem]{Corollary}
\newtheorem{claim}[theorem]{Claim}

\theoremstyle{remark}

\newcommand{\Z}{\mathbb{Z}}

\def \diam {\text{diam}}

\begin{document}

\title{Cartesian powers of graphs and consecutive radio labelings}
        \date{\today}

\author{Amanda Niedzialomski}

\maketitle
\begin{abstract}
For $k\in\mathbb{Z}^+$ and $G$ a simple connected graph, a $k$-radio labeling $f:V_G\to\Z^+$ of $G$ requires all pairs of distinct vertices $u$ and $v$ to satisfy $|f(u)-f(v)|\geq k+1-d(u,v)$.  When $k=1$, this requirement gives rise to the familiar labeling known as vertex coloring for which each vertex of a graph is labeled so that adjacent vertices have different ``colors".  We consider $k$-radio labelings of $G$ when $k=\diam(G)$.  In this setting, no two vertices can have the same label, so graphs that have radio labelings of consecutive integers are one extreme on the spectrum of possibilities.  Examples of such graphs of high diameter are especially rare and desirable.  We construct examples of arbitrarily high diameter, and explore further the tool we used to do this -- the Cartesian product of graphs -- and its effect on radio labeling.
\end{abstract}
\section{Introduction}
\subsection{Radio Labeling}
A radio labeling of a graph is a generalization of vertex coloring.  Given a graph $G$ (for us, $G$ is simple and connected) with vertex set $V_G$, any function $f:V_G\to\Z^+$ is a labeling of the vertices.  If that labeling satisfies the inequality
$$|f(u)-f(v)|\geq2-d(u,v)$$
for all distinct vertices $u,v\in V_G$, then $f$ is a coloring of $G$.  This inequality is exactly the condition needed to guarantee that adjacent vertices are labeled with a different ``color".  We can generalize this notion of graph coloring by changing the inequality to the following:
\begin{equation}
|f(u)-f(v)|\geq k+1-d(u,v)\label{rc}
\end{equation}
for some $k\in\Z$, $1\leq k\leq\diam(G)$.  A labeling satisfying this inequality for all vertices $u,v$ is called a $k$-radio labeling \cite{Chartrand}.  In particular, it guarantees that any two vertices that are $k$ apart or less in distance have different labelings (which is the reason that $k$ is bounded above by $\diam(G)$).  When $k=\diam(G)$, we call a $k$-radio labeling simply a radio labeling, and we call inequality \eqref{rc} the radio condition.

Radio labeling is connected to the Channel Assignment Problem, which aims to assign frequencies (channels) to radio transmitters in a way that minimizes interference.  So, if two transmitters are geographically close, it is optimal to have their frequencies be far apart.  This directly relates to the idea of radio labeling of a graph by thinking of the vertices as radio transmitters and the labels as radio frequencies (hence the name) \cite{Georges}, \cite{Hale}, \cite{Heuvel}.

Just as with vertex coloring, we are interested in minimizing the largest label given to a vertex.  For any particular labeling, the largest label used is called the span of the labeling.  In vertex coloring, the minimum possible span of a graph is called its chromatic number.  In radio labeling, the minimum span is called the radio number, denoted $\textrm{rn}(G)$.  Finding a closed formula for the radio numbers of different types of graphs is the goal, but it remains unknown for all but a handful of them \cite{Benson}, \cite{Fernandez}, \cite{Li}, \cite{Liu2}, \cite{Liu1}.

Even the discrete problem of finding the radio number of a specific example $G$ can be difficult because of the size of the computation.  One way of finding the radio number involves analyzing all possible permutations of the vertices.  Given a radio labeling $f$ of a graph with vertices $\{v_1, \dots, v_n\}$, we can rename (order) the vertices as $x_1, \dots, x_n$ so that $f(x_i)<f(x_j)$ if and only if $i<j$.  On the other hand, we can produce a radio labeling of a graph by picking an ordering of the vertices and then forcing the labels to increase with respect to that order.  Since we are interested in minimizing span, we always choose the smallest label that works at the given moment.  Therefore, an ordering of the vertices uniquely determines a radio labeling if we enforce this increasing condition.  It is this radio labeling we refer to when we speak of the radio labeling ``induced'' by the ordering.  So, one brute force way of computing the radio number of a graph is to consider all radio labelings induced from all possible orderings (permutations) of the vertices, and then choose the smallest span of all of them.  Even after an ordering is fixed, producing the induced radio labeling is nontrivial, so producing $|V_{G}|!$ labelings is computationally very intense.

\subsection{Consecutive Radio Labeling}
Unlike all other $k$-radio labelings of a graph, if $k=\diam(G)$, then no two distinct vertices can have the same label; this tells us that $\textrm{rn}(G)\geq|V_G|$.  If $\textrm{rn}(G)=|V_G|$, then there is a radio labeling of $G$ of consecutive integers $1,\dots,|V_G|$.  Such a labeling is called a consecutive radio labeling.  The complete graph $K_n$ of $n$ vertices is an example of a graph that has such a labeling.   Since $\diam(K_n)=1$, the radio condition is trivially satisfied, so any labeling with consecutive integers will work.  A more interesting example is that of the Peterson graph (Figure \ref{fig:PetersonGraph}).  Its diameter is 2, so the radio condition is not satisfied for all possible consecutive labelings, but it does have a consecutive radio labeling.  

\begin{figure}[h]
\begin{center}
\includegraphics[scale=.5]{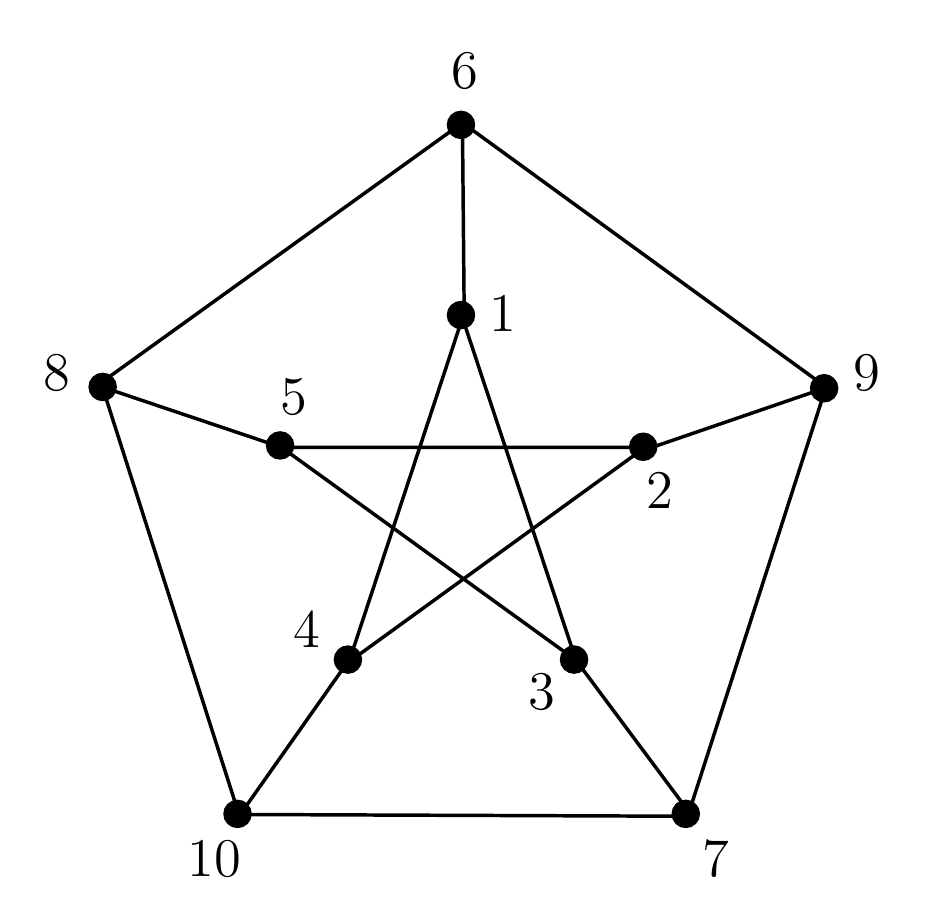}
\caption{The Peterson graph with a consecutive radio labeling}
\label{fig:PetersonGraph}
\end{center}
\end{figure}

The qualities that the Peterson Graph possesses that allow it to have a consecutive radio labeling can be boiled down to the following: it is a diameter 2 graph that has an ordering of its vertices such that each vertex is diameter away from the next vertex in the ordering.  The existence of such an ordering of the vertices is a rare property for a graph to have, but this alone will not always be sufficient to produce a graph with a consecutive radio labeling.  It is only because we also had a graph of diameter 2 in this case that the qualifications were so simple.

In general, for a graph $G$ with vertices $\{v_1,v_2,\dots,v_n\}$, we would need an ordering of the vertices $x_1,x_2,\dots, x_n$ such that 
$$d(x_i,x_{i+c})\geq\diam(G)-c+1.$$
This gives an idea of how special graphs that have consecutive radio labelings are, and how difficult it can be to produce examples of them, particularly examples of high diameter.

\subsection{Cartesian Product of Graphs}
A tool we will be using in our construction is the Cartesian product of graphs.  Given two graphs $G$ and $H$, their Cartesian product, denoted $G\ \Box \ H$, is a graph defined by the following:
\begin{itemize}
\item[(1)] The vertex set $V_{G\Box H}$ is given by $V_G\times V_H$.
\item[(2)]  Two vertices $(u,v),(u',v')\in V_{G\Box H}$ are adjacent if the vertices in one coordinate are the same, and the vertices in the other coordinate are adjacent in their respective original graph.  That is, if
\begin{itemize}
\item[(i)] $u=u'$ and $v$ is adjacent to $v'$ in $H$, or
\item[(ii)] $v=v'$ and $u$ is adjacent to $u'$ in $G$.
\end{itemize}
\end{itemize} 
See Figure \ref{fig:K3BoxP3} for an example.

\begin{figure}[h]
\begin{center}
\includegraphics[scale=.6]{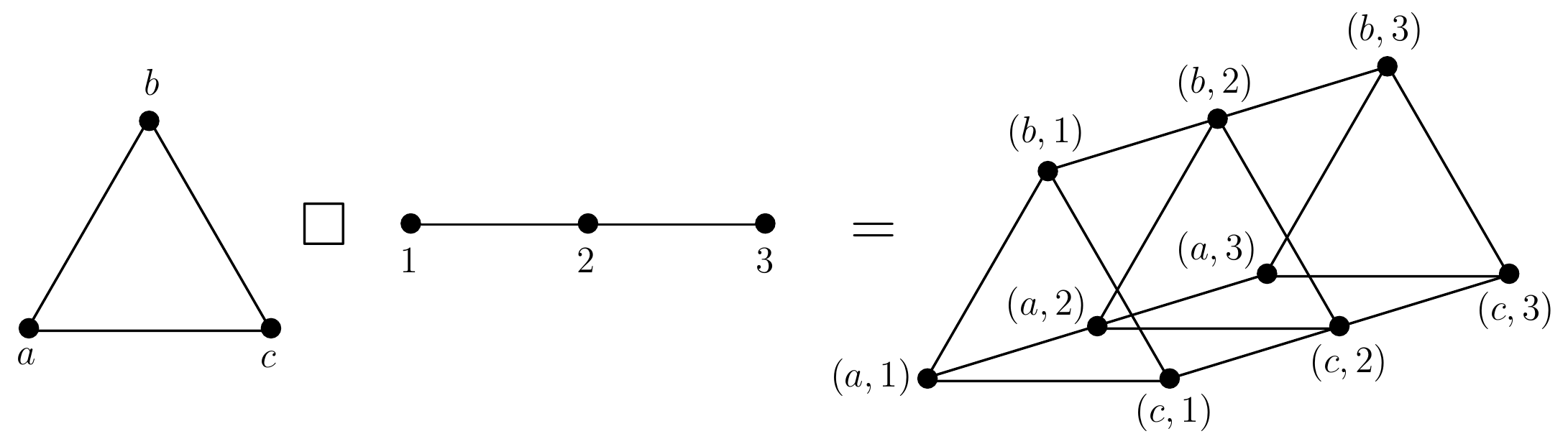}
\caption{Cartesian product example: $K_3\ \Box\ P_3$}
\label{fig:K3BoxP3}
\end{center}
\end{figure}

One very nice property of the Cartesian product of graphs is that distance is inherited from the original graphs: 
$$d((u,v),(u',v'))=d(u,u')+d(v,v').$$
This allows us to learn about $\textrm{rn}(G\ \Box\ H)$ by looking at $\textrm{rn}(G)$ and $\textrm{rn}(H)$.  We will be considering the Cartesian product of $t$ copies of a graph $G$, denoted $G^t$:
$$G^t=\underbrace{G\ \Box\ G\ \Box\cdots\Box\ G}_{t\text{ copies of }G}.$$
Vertices of $G^t$ can be represented as $t$-tuples with entries in $V_G$, and there is an edge between two vertices $(v_{i_1},v_{i_2},\dots,v_{i_t})$ and $(v_{j_1},v_{j_2},\dots,v_{j_t})$ if, for some $a\in\{1,\dots,t\}$, $v_{i_a}$ is adjacent to $v_{j_a}$ in $G$, and $v_{i_k}=v_{j_k}$ for all $k\neq a$.

The graph of $G^t$ can be quite complicated, even if $G$ itself is simple.  (The number of vertices of $G^t$ is $|V_G|^t$ and, if $E_G$ denotes the set of edges in $G$, then $G^t$ has $t|V_G|^{t-1}|E_G|$ edges.)

Distance is given by 
$$d_{G^t}((v_{i_1},v_{i_2},\dots,v_{i_t}),(v_{j_1},v_{j_2},\dots,v_{j_t}))=\sum_{k=1}^td_G(v_{i_k},v_{j_k}).$$
Diameter is then easily computed: $\diam(G^t)=\sum_{k=1}^{t}\diam(G)=t\cdot\diam(G)$.  

When $G$ has a consecutive radio labeling, we will be looking at $G^t$ for new examples of graphs with such a labeling.  Since $\diam(G^t)$ increases as $t$ does, we have some hope of finding high diameter examples by looking here.  Actually, one example is well known.  We have already seen that the Peterson graph (which we will hereafter denote $P$) has a consecutive radio labeling  (Figure \ref{fig:PetersonGraph}).  A consecutive radio labeling can also be exhibited for the diameter 4 graph $P^2$.  This finding is  exciting for the cause for finding higher diameter examples, and it can spark some (very optimistic) conjectures.  For example, it is tempting to hope that $P^t$ has a consecutive radio labeling for all $t\in\Z^+$.  This is desirable because if $t$ is unbounded above then so is $\diam(P^t)$.  It isn't true, however, as we will see in $\S$\ref{highPowers}.  Nevertheless, we were able to prove an even better result in the following section.  Better because it produces a family of examples with consecutive radio labelings in which there is an example for every possible diameter (whereas $P^t$ has even diameter).

\section{Consecutive radio labelings for certain powers of $K_n$}
We construct a family of examples of arbitrary diameter that have consecutive radio labelings by taking Cartesian powers of the complete graph $K_n$.  We will begin by giving an ordering\footnote{The definition of the ordering we give in this section allows a mental picture to develop that we will utilize.  Alternatively, the same ordering can be described as follows:  If $n\geq3$ and $1\leq t\leq n$, then this recursive definition is an ordering on $V_{K_n^t}$:
$$x^1_i=v_i,\text{ and, for }k=1,\dots,n,\ x^t_{mn+k}=\left(x^{t-1}_{\lfloor m/n\rfloor n+k}, v_{k+m-\lfloor m/n\rfloor\ (\text{mod}\ n)}\right)$$
where we use the convention that $a\ (\text{mod}\ n)\in\{1,\dots,n\}$.
} of the vertices of $K_n^t$, and then prove that this ordering induces a consecutive radio labeling.

\subsection{Definition of $\mathbf{x_1,x_2,\dots,x_{n^t}}$}\label{orderingDef}
Let $V_{K_n}=\{v_1, \dots, v_n\}$ be the vertices of $K_n$, $n\geq3$.  Consider $K_n^t$ where $t\in\Z_{\leq n}^+$.  We describe the desired ordering of the vertices of $K_n^t$ in groups of $n$ vertices at a time.  The first $n$ vertices are the $t$-tuples
\begin{align*}
&x_1 = (v_1, v_1, \dots, v_1)\\
&x_2 = (v_2, v_2, \dots, v_2)\\
&\ \vdots\\
&x_n = (v_n, v_n, \dots, v_n).
\end{align*}
For the sake of easier reference, we will think of these $n$ vertices as an $n\times t$ matrix $A^1$ where $a^1_{i,j}$ is the $j$th coordinate of $x_i$.  That is, 
$$ 
A^1=\left[
\begin{array}{cccc}
v_1&v_1&\dots&v_1\\
v_2&v_2&\dots&v_2\\
\vdots\\
v_n&v_n&\dots&v_n\\
\end{array}\right].
$$
Let $\sigma\in S_{V_{K_n}}$ be the $n$-cycle $(v_1\ v_2\ \cdots\ v_n)$.  We produce a second matrix $A^2$ where
$$a^2_{i,j}=
\begin{cases}
\sigma(a^1_{i,j}) & \text{if }j=t\\
 a^1_{i,j} & \text{otherwise}
\end{cases}.
$$
We then form subsequent matrices by the following rule.  To produce $A^k$, we first determine $p$, the largest integer such that $k\equiv1\pmod{n^p}$.  Then $A^k$  is a matrix made up of entries
$$a^k_{i,j}=
\begin{cases}
\sigma(a^{k-1}_{i,j}) & \text{if }j=t-p\\
 a^{k-1}_{i,j} & \text{otherwise}
\end{cases}.
$$
Using this definition we create $n^{t-1}$ matrices of dimensions $n\times t$: $A^1, A^2, \dots, A^{n^{t-1}}$.  Each matrix $A^k$ is identical to the one that came before ($A^{k-1}$) except for a single column which differs by an application of $\sigma$.  The $n$ rows of each matrix (which total $n^t$ rows over all $n^{t-1}$ matrices) correspond to vertices of $K_n^t$, and tell us our ordering, $n$ vertices at a time.  Formally, if $i=bn+c$, $c\in\{1,2,\dots,n\}$, then 
\begin{equation}\label{xisdefinition}x_i=x_{bn+c}=(a^{b+1}_{c,1},a^{b+1}_{c,2},\dots,a^{b+1}_{c,t}),\end{equation}
the $c$th row of matrix $A^{b+1}$.  
\subsection{The $\mathbf{x_i}$'s Form an Ordering}
Our first goal is to show that this definition of $x_1,\dots, x_{n^t}$ is actually an ordering of the vertices of $K_n^t$.  Certainly $\{x_1,\dots,x_{n^t}\}\subseteq V_{K_n}$.  This coupled with the fact that $|V_{K_n}|=n^t$ means we only need to show that our proposed ordering has no repetition: $x_i\neq x_j$ for all $i\neq j$.  We can make this task a bit easier by noticing some structure of the $A^k$'s.  $A^1$ has a structure that is inherited by  $A^k$ for all $k$: $a_{i,j}^k=\sigma(a_{i-1,j}^k)$.  Therefore, any row of $A^k$ determines all of $A^k$.  For this reason, it suffices to show that no two matrices have the same first row.

Since we are narrowing our scope to the first rows of each matrix, let's make a new matrix $A$ comprised of them.  Let $A$ be the $n^{t-1}\times t$ matrix defined by $a_{i,j}=a^i_{1,j}$. We can give an equivalent definition by using the definition of $A^k$.  To define $a_{i,j}$, let $q$ be the largest integer such that $i\equiv 1\pmod{n^q}$.  Then
$$a_{i,j}=
\begin{cases}
v_1 & \text{if }i=1\\
\sigma(a_{i-1,j}) & \text{if } j=t-q\\
a_{i-1,j} & \text{otherwise} 
\end{cases}.
$$
So now we show that $A$ has no identical rows.  We will start by studying the specific and repetitive structure of $A$.  To do this, we break up each column into ``blocks'' -- what we'll call a certain collection of column vectors that, when adjoined, produce the original column.  Fix a column $j$.  We wish to break up the column into groups of $n^{t-j}$ entries which will form the blocks.  So, the vectors
\begin{equation}\label{jblocks}
\left[\begin{array}{c}
a_{1,j}\\
\vdots\\
a_{n^{t-j},j}
\end{array}\right],\left[\begin{array}{c}
a_{n^{t-j}+1,j}\\
\vdots\\
a_{2n^{t-j},j}
\end{array}\right],\dots,\left[\begin{array}{c}
a_{(n^{j}-1)n^{t-j}+1,j}\\
\vdots\\
a_{n^{t},j}
\end{array}\right] 
\end{equation}
are the blocks of column $j$.  Note that the blocks are of uniform size, but that size is dependent on the column.  The first column of $A$ has only one block, while in the last column of $A$ each entry is itself a block. In order to easily reference which column a block comes from, we will commonly call a block from column $j$ a $j$-block.  In summary:
\begin{definition}
A \textbf{j-block} is any one of the vectors listed in line \eqref{jblocks}.
\end{definition}
We will call the set of rows that are associated to the entries of a block the scope of the block.
\begin{definition}
The \textbf{scope} of a $j$-block $\left[\begin{array}{c}
a_{cn^{t-j}+1,j}\\
\vdots\\
a_{(c+1)n^{t-j},j}
\end{array}\right]$ is the set with rows $cn^{t-j}+1$ through $(c+1)n^{t-j}$ of $A$ as its elements: $\{[a_{cn^{t-j+1},h}]_{1\leq h\leq t},\dots,[a_{(c+1)n^{t-j},h}]_{1\leq h\leq n}\}$.  The scope of multiple $j$-blocks is the union of the scopes of the individual $j$-blocks.
\end{definition}

\begin{claim}\label{blockshaveidenticalentries}
Given a $j$-block of $A$ that has entries $a_{cn^{t-j}+1,j}, \dots, a_{(c+1)n^{t-j},j}$, then $a_{cn^{t-j}+1,j} = \dots = a_{(c+1)n^{t-j},j}$.
\end{claim}
\begin{proof}
Based on the definition of $A$, $a_{i,j}=a_{i-1,j}$ unless $j=t-q$ where $q$ is the largest integer such that $i\equiv1\pmod{n^q}$.  Equivalently, $a_{i,j}=a_{i-1,j}$ unless $q=t-j$.  The only way to get any distinct entries in the block is for one of the entries (excluding the first) to be the image under $\sigma$ of the previous entry.  So the question is, is it possible for $q=t-j$ when $i\in\{cn^{t-j}+2, cn^{t-j}+3, \dots, (c+1)n^{t-j}\}$?  One condition for $q=t-j$ is that $i\equiv1\pmod{n^{t-j}}$.  It's clear that this cannot be the case for any of the possibilities for $i$:
\begin{align*}
cn^{t-j}+2&\equiv2\pmod{n^{t-j}}\\
cn^{t-j}+3&\equiv3\pmod{n^{t-j}}\\
&\ \vdots\\  
(c+1)n^{t-j}&\equiv0\pmod{n^{t-j}}.
\end{align*}
Therefore, since $j\neq t-q$ for all $i\in\{cn^{t-j}+2, cn^{t-j}+3, \dots, (c+1)n^{t-j}\}$, each entry of the vector must be the same as the first entry $a_{cn^{t-j}+1,j}$.
\end{proof}
As a result of this claim, along with the definition of $A$, we know that adjacent blocks in a column can only have two possible relationships: either they are identical, or the second block has entries which are the image under $\sigma$ of the first block's entries.  We'd like to know when adjacent blocks are identical.
\begin{claim}\label{blockrepetition}
The adjacent blocks in column $j$ of $A$
$
\left[\begin{array}{c}
a_{cn^{t-j}+1,j}\\
\vdots\\
a_{(c+1)n^{t-j},j}
\end{array}\right],
\left[\begin{array}{c}
a_{(c+1)n^{t-j}+1,j}\\
\vdots\\
a_{(c+2)n^{t-j},j}
\end{array}\right]
$ are identical if and only if $c+1$ divides $n$.
\end{claim}
\begin{proof}
It follows from Claim \ref{blockshaveidenticalentries} that the blocks are identical if and only if $a_{(c+1)n^{t-j},j}=a_{(c+1)n^{t-j}+1,j}$.  We know from our definition of $A$ that this happens only when $j\neq t-q$, where $q$ is computed for $i=(c+1)n^{t-j}+1$.  So what is $q$ in this case?  Certainly $t-j$ is a candidate since $(c+1)n^{t-j}+1\equiv1\pmod{n^{t-j}}$.  In fact, $q\neq t-j$ if and only if $q>t-j$, which is equivalent to requiring that $c+1$ divides $n$.
\end{proof}
The blocks have now given us a really detailed picture of what's happening with our matrix $A$.  If we look down an arbitrary column $j$, the column is partitioned into $n^{j-1}$ blocks, each with $n^{t-j}$ entries.  The entries of any fixed block are identical.  The first block in each column has all entries of $v_1$.  When changes do happen down a column, they change according to $\sigma$; so $v_1$ can be followed by $v_2$ then $v_3$ and so on.  And we now know by Claim \ref{blockrepetition}  exactly when those changes will not occur due to the repetition of adjacent blocks.  Now we want to examine how the structure of the different columns relate.

It happens that the scope of a ($j-1$)-block is equal to the scope of a collection of $j$-blocks -- the appropriate $n$ of them in a row.  This is immediate because of the size of the blocks.  A ($j-1$) block has $n^{t-j+1}$ entries, and $j$-blocks have $n^{t-j}$ entries.  So, for every block in column $j-1$, there are ${n^{t-j+1}}/{n^{t-j}}=n$ blocks in column $j$.  The interesting thing is that those $n$ $j$-blocks with the same scope as the single ($j-1$)-block must be distinct (in the sense that no two of these $j$-blocks can have the same entries).  We will prove this in the following claim.  Note that the blocks listed in the claim are in column $j$, of the correct size ($n^{t-j}$ entries), they are adjacent, and there are $n$ of them.  Also note that they cover rows $bn^{t-j+1}+1$ through $(b+1)n^{t-j+1}$.  This is significant because $\left[\begin{array}{c}
a_{bn^{t-j+1}+1,j+1}\\
\vdots\\
a_{(b+1)n^{t-j+1},j+1}
\end{array}\right] $ is a ($j-1$)-block.
\begin{claim}\label{distinctblocks}
The following $j$-blocks in $A$ are all distinct:
$$
\left[\begin{array}{c}
a_{bn^{t-j+1}+1,j}\\
\vdots\\
a_{bn^{t-j+1}+n^{t-j},j}
\end{array}\right],\left[\begin{array}{c}
a_{bn^{t-j+1}+n^{t-j}+1,j}\\
\vdots\\
a_{bn^{t-j+1}+2n^{t-j},j}
\end{array}\right],\dots,\left[\begin{array}{c}
a_{bn^{t-j+1}+(n-1)n^{t-j}+1,j}\\
\vdots\\
a_{(b+1)n^{t-j+1},j}
\end{array}\right] 
.$$
\end{claim}
\begin{proof}
Let's begin by proving that the first two $j$-blocks $
\left[\begin{array}{c}
a_{bn^{t-j+1}+1,j}\\
\vdots\\
a_{bn^{t-j+1}+n^{t-j},j}
\end{array}\right],\left[\begin{array}{c}
a_{bn^{t-j+1}+n^{t-j}+1,j}\\
\vdots\\
a_{bn^{t-j+1}+2n^{t-j},j}
\end{array}\right]$ 
are distinct.  In order to easily reference Claim \ref{blockrepetition}, we rewrite the indices of the entries in the form that appears there:
$$
\left[\begin{array}{c}
a_{bn^{t-j+1}+1,j}\\
\vdots\\
a_{bn^{t-j+1}+n^{t-j},j}
\end{array}\right]=\left[\begin{array}{c}
a_{(bn)n^{t-j}+1,j}\\
\vdots\\
a_{(bn+1)n^{t-j},j}
\end{array}\right],\left[\begin{array}{c}
a_{bn^{t-j+1}+n^{t-j}+1,j}\\
\vdots\\
a_{bn^{t-j+1}+2n^{t-j},j}
\end{array}\right]=
\left[\begin{array}{c}
a_{(bn+1)n^{t-j+1}+1,j}\\
\vdots\\
a_{(bn+2)n^{t-j+1},j}
\end{array}\right].$$
Now we are already done because (by Claim \ref{blockrepetition}) showing these two blocks are distinct is equivalent to showing $n$ does not divide $bn+1$, which is immediate.

With the same argument we can show that, in this collection of $n$ blocks, adjacent blocks are distinct.  We proceed by similarly rearranging the remaining indicies to be compatible with the previous claim.
\begin{align*}
\left[\begin{array}{c}
a_{bn^{t-j+1}+2n^{t-j}+1,j}\\
\vdots\\
a_{bn^{t-j+1}+3n^{t-j},j}
\end{array}\right]&=\left[\begin{array}{c}
a_{(bn+2)n^{t-j}+1,j}\\
\vdots\\
a_{(bn+3)n^{t-j},j}
\end{array}\right]\\
&\ \vdots\\
\left[\begin{array}{c}
a_{bn^{t-j+1}+(n-1)n^{t-j}+1,j}\\
\vdots\\
a_{(b+1)n^{t-j+1},j}
\end{array}\right]&=\left[\begin{array}{c}
a_{(bn+n-1)n^{t-j}+1,j}\\
\vdots\\
a_{(bn+n)n^{t-j},j}
\end{array}\right].
\end{align*}
Showing that adjacent blocks are distinct is equivalent to showing that $n$ does not divide $bn+2, bn+3, \dots, bn+n-1$, which is again immediate.  Because of this we can now say the following: if the first block has entries equal to $v_k$, then the second has entries of $\sigma(v_k)$, and the third has entries of $\sigma^2(v_k)$, $\dots$, and the $n$th has entries of $\sigma^{n-1}(v_k)$.  Since $\sigma$ is an $n$-cycle, $v_k\neq\sigma(v_k)\neq\sigma^2(v_k)\neq\cdots\neq\sigma^{n-1}(v_k)$, which confirms that the $n$ blocks are distinct.
\end{proof}
We are now able to prove matrix $A$ has no repeated rows by showing that an arbitrary row can only exist in one place in $A$.
\begin{lemma}
The matrix $A$ has no two rows with identical entries.
\end{lemma}
\begin{proof}
We begin by proving the following claim using an induction argument.\\
\textbf{Claim:} If two rows share their first $k$ entries, then they both belong to the scope of the same $k$-block.  \emph{Pf}: As we've already mentioned, the first column of $A$ is itself one block, so the $k=1$ case is trivial.  Suppose the claim is true for $k$, and let the $x$th and $y$th rows share their first $k+1$ entries.  Then they must share their first $k$ entries, so by assumption rows $x$ and $y$ belong to the scope of a single $k$-block.  By Claim 3, this scope is equal to that of a collection of $n$ distinct ($k+1$)-blocks.  Since rows $x$ and $y$ also share the ($k+1$)st entry, they can only be in the scope of one of those ($k+1$)-blocks, which proves the claim. 

This result tells us that if rows $x$ and $y$ share all $t$ of their entries, then they belong to the scope of the same $t$-block.  But a $t$-block consists of $n^{t-t}=1$ entry, so its scope contains only one row.  Then $x=y$ and no two rows of $A$ can have identical entries.
\end{proof}

\begin{theorem}\label{orderingofknt}
Let $n\in\Z_{\geq3}$, $t\in\Z^+_{\leq n}$.  As defined in line \eqref{xisdefinition}, $x_1,\dots,x_{n^t}$ is an ordering of the vertices of $K_n^t$.
\end{theorem}
Over the course of this section, we have already made the arguments proving Theorem \ref{orderingofknt}.  Let's review.  We defined the $x_i$'s to be the rows of the matrices $A^1,\dots,A^{n^{t-1}}$.  We needed only to show that no two of those rows are identical.  For each $k$, $A^k$ has the characterisitic that $a^k_{i,j}=\sigma(a^k_{i-1,j})$, which implies that fixing a row fixes all of $A^k$.  (It also implies that $A^k$ cannot have any repeated rows itself.)  So if any two rows across all the rows of $A^1,\dots,A^{n^{t-1}}$ are identical, that would imply two of the $A^k$'s are identical.  Then of course the first rows of those two matrices are identical, which makes two rows in $A$ identical, but we just proved that this cannot happen.  So our definition of $x_1,\dots,x_{n^t}$ is indeed an ordering of $V_{K_n^t}$.

\subsection{The Ordering Induces a Consecutive Radio Labeling}
To prove that our ordering induces a consecutive radio labeling, we need to keep track of instances where ``close" vertices have common entries.  Let $e_{i,j}$ be the number of coordinates where $x_i$ and $x_j$ agree.  If $x_i=(v_{i_1},v_{i_2},\dots,v_{i_t})$ and $x_j=(v_{j_1},v_{j_2},\dots,v_{j_t})$, then
$$d_{K_n^t}(x_i,x_j)=\sum_{k=1}^td_{K_n}(v_{i_k},v_{j_k})=t-e_{i,j}.$$
Let $f$ be the radio labeling induced from our ordering.  The radio condition becomes
$$|f(x_i)-f(x_j)|\geq \diam(K_n^t)+1-d(x_i,x_j)=t+1-(t-e_{i,j})=1+e_{i,j}.$$
Equivalently, $e_{i,j}\leq|f(x_i)-f(x_j)|-1$.  In order for $f$ to be consecutive, we need $f(x_i)=i$ for all $i$, so we need $e_{i,j}$ to satisfy
\begin{equation}\label{eijcondition}e_{i,j}\leq|i-j|-1.\end{equation}
Since $e_{i,j}$ counts coordinates, $0\leq e_{i,j}\leq t$.  So $e_{i,i+s}$ trivially satisfies \eqref{eijcondition} if $s\geq t+1$.  Also, if $s=t$, then \eqref{eijcondition} is equivalent to requiring that $x_i$ and $x_{i+s}$ are not identical, which we've already proven must be the case.  We will show that $e_{i,i+s}$ satisfies \eqref{eijcondition} for all $s$.
\begin{theorem}\label{KntConsecutive}
Let $n\in\Z_{\geq3}$, $t\in\Z^+_{\leq n}$.  Then $K^t_n$ has a consecutive radio labeling.
\end{theorem}
\begin{proof}
We use the ordering $x_1,\dots,x_{n^t}$ of the vertices of $K_n^t$ from Theorem \ref{orderingofknt}, and we let $f:V_{K_n^t}\to\Z^+$ be defined as $f(x_i)=i$.  We will prove that $f$ satisfies the radio condition.  From the above discussion, we need to show that $e_{i,i+s}$ satisfies $e_{i,i+s}\leq|i-(i-s)|-1=s-1$ for all $1\leq s\leq t-1$.  

We will be referring again to the $A^k$'s -- the matrices from $\S$\ref{orderingDef} used to define $x_1,\dots,x_{n^t}$.  Let $x_i$ be a row in matrix $A^{k_i}$.  If $x_{i+s}$ is also a row in matrix $A^{k_i}$, then $e_{i,i+s}=0$ since $a_{i,j}^{k_i}=\sigma(a_{i-1,j}^{k_i})$.

Suppose $x_{i+s}$ is not a row in $A^{k_i}$.  By assumption $s\leq t-1< n$.  Since the $A^{k}$'s each have $n$ rows, $x_{i+s}$ must be a row in $A^{k_i+1}$.  Note that, because $s<n$, $x_i$ and $x_{i+s}$ cannot have the same placement as rows in $A^{k_i}$ and $A^{k_i+1}$ (i.e., if $x_i$ is the third row of $A^{k_i}$, then $x_{i+s}$ cannot be the third row of $A^{k_i+1}$).  We have seen that $A^{k_i}$ and $A^{k_i+1}$ are identical except for a single column, and that column differs by an application of $\sigma$.  These last two observations together imply that $e_{i,i+s}\leq1$.

With this, there is only one case left to consider: $s=1$.  We need to show $e_{i,i+1}=0$.  In this situation, $x_{i}$ must be the $n$th row of $A^{k_i}$ and $x_{i+s}$ must be the first row of $A^{k_i+1}$.  So $x_{i+1}$ is identical to the first row of $A^{k_i}$ except for one coordinate, let's call this the $c$th coordinate.  That is,
$$x_{i+1}=(a_{1,1}^{k_i},\dots,\sigma(a_{1,c}^{k_i}),\dots,a_{1,t}^{k_i}).$$  We also know that $x_i=(\sigma^{n-1}(a_{1,1}^{k_i}),\dots,\sigma^{n-1}(a_{1,c}^{k_i}),\dots,\sigma^{n-1}(a_{1,t}^{k_i}))$.  We can observe then that $x_i$ and $x_{i+1}$ share no common coordinates (because $\sigma$ is an $n$-cycle, $n\geq3$).  In particular, we check that the $c$th coordinates are different: indeed, $\sigma(a_{1,c}^{k_i})\neq\sigma^{n-1}(a_{1,c}^{k_i})$.  We have shown $e_{i,i+1}=0$.  Therefore, $e_{i,i+s}\leq s-1$ for all $s$, so $K_n^t$ has a consecutive radio labeling.
\end{proof}
\section{High powers of $G$}\label{highPowers}
It's natural to now wonder if $K_n^t$ has a consecutive radio labeling for all $t\in\Z^+$.  It turns out that we can find a high enough power $t$ such that $K_n^t$ no longer has a consecutive radio labeling.  With this next theorem we can do better -- we can find such a $t$ for an arbitrary graph $G$. 
\begin{theorem}{Given a graph $G$, there is an integer $s$ such that $G^t$ does not have a consecutive radio labeling for any $t\geq s$.  In particular, if $G$ has $n$ vertices, 
$$s=1+\sum_{i=\text{\emph{diam}}(G)}^{n-1}(n-i)\left\lfloor\frac{i}{\text{\emph{diam}}(G)}\right\rfloor$$ is such a value.}
\end{theorem}
\begin{proof}
Let $V_{G} = \{v_1,\dots, v_n\}$.  Let $s=1+\sum_{i=\text{\textrm{diam}}(G)}^{n-1}(n-i)\left\lfloor\frac{i}{\text{\textrm{diam}}(G)}\right\rfloor$, and let $t\in\Z$, $t\geq s$.  In search of contradiction, suppose $x_1,\dots,x_{n^t}$ is an ordering of the vertices such that $f:V_{G^t}\to \Z^+$ defined by $f(x_i)= i$ is a radio labeling.  Then $f$ must satisfy the radio condition for all $x_i, x_j\in V_{G^t}$:
\begin{equation}\label{powerDistance}
d(x_i, x_j)\geq\diam(G^t)-|f(x_i)-f(x_j)|+1=t\cdot\diam(G)-|i-j|+1.
\end{equation}

We can get a get an upper bound for $d(x_i, x_j)$ by counting the number of coordinates for which $x_i$ and $x_j$ have common entries.  If we define $e_{i,j}$ to be this number of coordinates where $x_i$ and $x_j$ agree, then 
\begin{equation}\label{powerDistance}
d(x_i,x_j)\leq\diam(G)(t-e_{i,j}).
\end{equation}
Combining these two bounds on $d(x_i,x_j)$ yields
\begin{align*}
t\cdot\diam(G)-|i-j|+1&\leq\diam(G)(t-e_{i,j})\\
t-\frac{|i-j|-1}{\diam(G)}&\leq t-e_{i,j}\\
e_{i,j}&\leq\frac{|i-j|-1}{\diam(G)}.
\end{align*}
Then the maximum possible number of coordinates with common entries between $x_i$ and $x_j$ is
$$E_{i,j}=\min\left\{t,\left\lfloor\frac{|i-j|-1}{\diam(G)}\right\rfloor\right\}.$$
Fix $x_i$.  The maximum number of coordinates that $x_i$ can have in common with any prior vertex is $\sum_{j=1}^{i-1}E_{i,j}$.  Then the total number of coordinates in which any of the first $n+1$ vertices in the ordering can agree is
\begin{equation*}
\underbrace{\ \sum_{j=1}^{1}E_{2,j}\ }_{
\begin{smallmatrix}
\text{number of coordinates}\\x_2\text{ can have in common}\\\text{with any prior vertex}
\end{smallmatrix}
}
+\underbrace{\ \sum_{j=1}^{2}E_{3,j}\ }_{
\begin{smallmatrix}
\text{number of coordinates}\\x_3\text{ can have in common}\\\text{with any prior vertex}
\end{smallmatrix}
}
+\ \ \ \cdots\ \ \ +\underbrace{\ \sum_{j=1}^{n}E_{n+1,j}\ }_{
\begin{smallmatrix}
\text{number of coordinates}\\x_{n+1}\text{ can have in common}\\\text{with any prior vertex}
\end{smallmatrix}
}
=\sum_{i=2}^{n+1}\sum_{j=1}^{i-1}E_{i,j}.
\end{equation*}
We can obtain with some computation that
{\small\begin{align*}
\sum_{i=2}^{n+1}\sum_{j=1}^{i-1}E_{i,j}&=\sum_{i=2}^{n+1}\sum_{j=1}^{i-1}\left\lfloor\frac{|i-j|-1}{\diam(G)}\right\rfloor\\
&=\cancel{\sum_{j=1}^1\left\lfloor\frac{|2-j|-1}{\diam(G)}\right\rfloor}+\sum_{i=3}^{n+1}\sum_{j=1}^{i-1}\left\lfloor\frac{|i-j|-1}{\diam(G)}\right\rfloor\\
&=\sum_{j=1}^2\left\lfloor\frac{|3-j|-1}{\diam(G)}\right\rfloor+\sum_{j=1}^3\left\lfloor\frac{|4-j|-1}{\diam(G)}\right\rfloor+\cdots+\sum_{j=1}^n\left\lfloor\frac{|n+1-j|-1}{\diam(G)}\right\rfloor\\
&=\sum_{j=1}^2\left\lfloor\frac{2-j}{\diam(G)}\right\rfloor+\sum_{j=1}^3\left\lfloor\frac{3-j}{\diam(G)}\right\rfloor+\cdots+\sum_{j=1}^n\left\lfloor\frac{n-j}{\diam(G)}\right\rfloor\\
&=\sum_{i=1}^1\left\lfloor\frac{i}{\diam(G)}\right\rfloor+\sum_{i=1}^2\left\lfloor\frac{i}{\diam(G)}\right\rfloor+\cdots+\sum_{i=1}^{n-1}\left\lfloor\frac{i}{\diam(G)}\right\rfloor\\
&=(n-1)\left\lfloor\frac{1}{\diam(G)}\right\rfloor+(n-2)\left\lfloor\frac{2}{\diam(G)}\right\rfloor+\cdots+\left\lfloor\frac{n-1}{\diam(G)}\right\rfloor\\
&=\sum_{i=1}^{n-1}(n-i)\left\lfloor\frac{i}{\diam(G)}\right\rfloor=\sum_{i=\diam(G)}^{n-1}(n-i)\left\lfloor\frac{i}{\diam(G)}\right\rfloor=s-1<t.
\end{align*}}

Since $t$ is larger than the number of coordinates with this property, there is at least one coordinate in which none of the vertices $x_1,\dots,x_{n+1}$ agree.  This is impossible to accomplish, however, as we have only $n$ possible entries: $v_1,\dots, v_n$.  So there is no ordering of the $n^t$ vertices of $G^t$ that induces a consecutive radio labeling.
\end{proof}

\begin{corollary}\label{KntForHight}
$K_n^t$ does not have a consecutive labeling for any $t\geq1+\frac{n(n^2-1)}{6}$.
\end{corollary}
\begin{proof}
This is a simple computation:
\begin{align*}
s=1+\sum_{i=1}^{n-1}(n-i)\left\lfloor\frac{i}{1}\right\rfloor&=1+\sum_{i=1}^{n-1}ni-i^2\\
&=1+n\sum_{i=1}^{n-1}i-\sum_{i=1}^{n-1}i^2\\
&=1+n\left(\frac{n(n-1)}{2}\right)-\frac{n(n-1)(2n-1)}{6}\\
&=1+\frac{n(n^2-1)}{6}.
\end{align*}
\end{proof}

\section{Open Questions}
Together, Theorem \ref{KntConsecutive} and Corollary \ref{KntForHight} give us a lot of information about $K_n^t$.  If $1\leq t\leq n$, then $K_n^t$ has a consecutive radio labeling, so $\textrm{rn}(K_n^t)=|V_{K_n^t}|=n^t$.  If $t\geq1+\frac{n(n^2-1)}{6}$ then $K_n^t$ does not have a consecutive radio labeling, so $\textrm{rn}(K_n^t)>n^t$.  This is illustrated in Figure \ref{fig:Kn_graph}.

\begin{figure}[h]
\begin{center}
\includegraphics[scale=1]{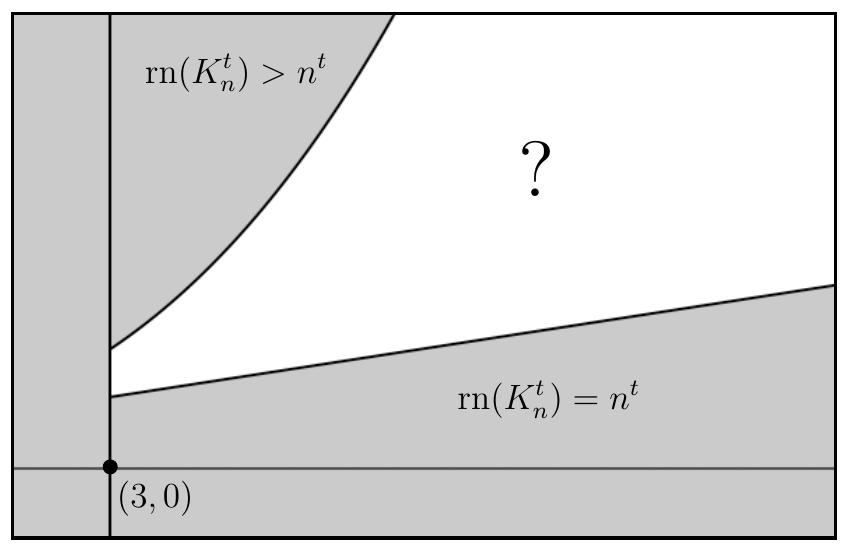}
\caption{An $n$ versus $t$ plot describing what we know (and what we don't) about $\textrm{rn}(K_n^t)$.}
\label{fig:Kn_graph}
\end{center}
\end{figure}

Determining the behavior of $K_n^t$ for $n<t<1+\frac{n(n^2-1)}{6}$ is one project that could be considered.

\end{document}